\documentclass[12pt,reqno]{amsart}

\usepackage{mathrsfs}
\usepackage{amsfonts}
\usepackage{amssymb}
\usepackage{amsthm}
\usepackage{amsmath}
\usepackage{latexsym}
\usepackage{tikz}
\usepackage{cite}
\allowdisplaybreaks

\newtheorem{theorem}{Theorem}[section]

\newtheorem{lemma}[theorem]{Lemma}

\numberwithin{equation}{section}

\newcommand{\beq}{\begin{equation}}
\newcommand{\eeq}{\end{equation}}

\newcommand{\Rmnum}[1]{\expandafter\@slowromancap\romannumeral #1@}

\newcommand{\ben}{\begin{eqnarray}}
\newcommand{\een}{\end{eqnarray}}
\newcommand{\beno}{\begin{eqnarray*}}
\newcommand{\eeno}{\end{eqnarray*}}

\begin{document}

\title[Global Regularity to Incompressible Viscoelastic ]
{Global Regularity to Incompressible Viscoelastic System With a Class of Large Initial Data}

\author[ Y. Zhu ]{ Yi Zhu}

\address{Department of Mathematics, East China University of Science and Technology, Shanghai 200237,  P.R. China}

\email{zhuyim@ecust.edu.cn}

\date{}
\subjclass[2010]{76A10, 76D03, 35B65}
\keywords{Incompressible viscoelastic flows, Global regularity, Large data}

\begin{abstract}
The global existence of solutions to incompressible viscoelastic flows has been a longstanding open problem, even for the global weak solution.
Under some special structure (``div-curl'' condition) the global small smooth solution was obtained in \cite{llzhou, cz}.
However, the result with large initial data remains unknown up to now, and it is studied in this paper.
We shall put forward a new structure: the cone-condition and then derive the global smooth solution to 3D incompressible viscoelastic system with a class of large initial data. The key is to gain an angle quantity in the estimate of nonlinear terms by restricting the solution to a cone.
\end{abstract}

\maketitle

\section{introduction}

\subsection{Background and Main Result}

The flow of three-dimensional (3D) incompressible viscoelastic system can be described by the following system
\begin{equation}\label{e1.1}
\left\{
  \begin{array}{ll}
    u_t +u\cdot\nabla u - \mu \Delta u +\nabla P = \nabla \cdot \Big[ \frac{\partial W(F)}{\partial F}F^T \Big], &  x \in \mathbb{R}^3,t>0, \\
    F_t +u\cdot\nabla F = \nabla uF,\\
    \nabla\cdot u=0 ,
  \end{array}
\right.
\end{equation}
where $u(t,x) \in \mathbb{R}^3$  represents the velocity field of materials, $P(t,x)$ means the scalar pressure. $F(t,x) \in \mathcal{M}$ denotes the deformation tensor ($\mathcal{M}$ is the set of $3\times 3$ matrices) and $W(F)$ is the elastic energy functional. The constant $\mu>0$ is the coefficient of viscosity.
In the first equation, $\frac{\partial W(F)}{\partial F}$ is the Piola-Kirchhoff stress tensor and $\frac{\partial W(F)}{\partial F}F^T$ is the Cauchy-Green tensor.
Throughout this paper we will adopt the following notations

$$ [\nabla u]_{ij}= \nabla_j u_i, \qquad [\nabla\cdot F]_i = \sum_j \nabla_j F_{ij}. $$

We consider here the Cauchy problem of system \eqref{e1.1}  with Hookean elasticity (namely $W(F)= |F|^2$) and the initial data will be specified by
\begin{equation}\label{initial}
(u,F)|_{t=0} = (u_0,F_0)(x) \qquad for\;all \quad x\in \mathbb{R}^3.
\end{equation}
Moreover the initial velocity filed satisfies the divergence free condition,
\begin{equation}\label{divergence}
 \nabla\cdot u_0 = 0,
\end{equation}
and the initial deformation tensor satisfies the natural constraint,
\begin{equation}\label{det}
  \det F_0 =1 .
\end{equation}

In this paper, we shall construct a family of global smooth solutions to the 3D incompressible viscoelastic system above with large initial data.

In the last several decades, there have been many attempts to capture different phenomena of the incompressible viscoelastic systems.
We mention some works here as a short review.
The global existence of a small smooth solution was established by Guillop\'{e} and Saut \cite{GS}, in which an additional linear damping term was presented in the equation of $F$.
In \cite{LM1} Lions and Masmoudi constructed global weak solutions for some Oldroyd-B type models.
And in \cite{CM1}, Chemin and Masmoudi proved the existence and uniqueness of local and global solutions in the critical Besov space.

If the damping term is absent, the velocity viscosity alone may not be sufficient to guarantee the regularity.
The study of near equilibrium dynamics of this system is both relevant and very important.
Here we consider the case deformation gradient $F$ is near an identity matrix and let $E=F-I$ be the perturbation. Then it's natural to derive the system for $(u,E)$ from \eqref{e1.1}.
Notice that the system satisfies the law of Hookean elasticity, there is

\begin{equation}\label{new}
\left\{
  \begin{array}{ll}
    u_t +u\cdot\nabla u - \mu \Delta u +\nabla P = \nabla \cdot (EE^T) + \nabla\cdot E + \nabla \cdot E^T,   x \in \mathbb{R}^3,t>0,\\
    E_t +u\cdot\nabla E = \nabla uE + \nabla u,\\
    \nabla\cdot u=0 .
  \end{array}
\right.
\end{equation}
Under the ``div'' structure assumption of initial data, i.e.,
\begin{equation}\label{div}
\nabla \cdot E_0^T =0,
\end{equation}
using basic energy analysis, one can find the divergence part $\nabla\cdot E$ is good enough while the curl part $\nabla\times E$ behaves roughly which brings the main difficulty.
In 2D case, Lin, Liu and Zhang \cite{llz} studied an auxiliary vector field and proved the global existence of small solutions under \eqref{div} (we refer to \cite{llzhou2} for different approach).
While in 3D case, other more structure is needed to handle the wildest part $\nabla\times E$.
In \cite{llzhou}, Lei, Liu and Zhou found a so-called ``curl'' structure of initial data,
\begin{equation}\label{curl}
\nabla_k (E_0)_{ij}-\nabla_j (E_0)_{ik}=(E_0)_{lj} \nabla_l (E_0)_{ik}-(E_0)_{lk}\nabla_l (E_0)_{ij},
\end{equation}
which is physical and compatible with the system.
Under both ``div'' structure and ``curl'' structure, they regarded $\nabla\times E$ as higher order terms, and then proved  the results for  global small solutions in 2D and 3D.
Chen and Zhang \cite{cz} picked another ``curl-free'' structure  $\nabla \times (F_0^{-1}-I)=0 $ in proving the Cauchy problem.
Under the same structure assumption,  the initial-boundary value results were done by Lin and Zhang \cite{lz}.
The result of critical $L^p$ framework was given by Zhang and Fang \cite{lp}. Feng, Zhu and Zi \cite{fzz} studied the blow up criterion. We also refer to \cite{fz, lzhou} for incompressible limit theory.
Recently, the author \cite{jfa} proved the global existence of small solutions to \eqref{e1.1} without any physical structure in 3D.
We  refer to \cite{fgo, hns, mt, old5, old6, old7, old2, old3, old4} for results considering the related Oldroyd-B models.
And refer to \cite{zhang, xuli, hu, Tan, com1, com2, hu3, yong, Gen, jjw, PX, PXZ, Zhu-Visco} for results about compressible viscoelastic systems.

All the above works are concerning with the small initial data problems, while for the large data problem the known results are few.
In \cite{hulin}, Hu and Lin studied the general data problem of 2D incompressible viscoelastic flows and obtained a class of global weak solution. However, the global existence of general weak solutions for incompressible viscoelastic system still remains a longstanding open problem even in two spatial dimensions.

On the other hand, if we focus on the global existence of smooth solutions to this system, to the best of our knowledge, the smallness of initial data is necessary.
In other words, there is few result considering the classical large data problem for system \eqref{e1.1} and it is studied in this paper.

The  purpose of this paper is to construct a new type of large solutions for the 3D incompressible viscoelastic system. The method in deriving the large solutions may also be effective for other incompressible fluids.

Before we state our main theorem, we want to introduce the so-called ``cone-structure'' which was first put forward by the author \cite{zy} to construct the large solution for 3D incompressible Navier-Stokes equations.
Firstly, let's define the cone $\Omega_{\theta_0,\xi_0}$ in $\mathbb{R}^3$ as follows:
\begin{equation}\nonumber
\Omega_{\theta_0,\xi_0} \doteq  \big\{   \xi \big|  \xi \in \mathbb{R}^3, \arg \langle \xi, \xi_0 \rangle \; \leq \theta_0 \big\}.
\end{equation}
Here notation $\arg \langle \xi, \xi_0 \rangle$ denotes the included angle (which is normalized to a real number in $[0, \pi)$) of vector $\xi$ and vector $\xi_0$.
Now, we can give the main result in the following theorem.

\begin{theorem}\label{t1.2}
Consider the Cauchy problem of 3D incompressible viscoelastic system \eqref{e1.1}-\eqref{initial} with Hookean elasticity. The smooth initial data $u_0, E_0$ satisfies \eqref{divergence}, \eqref{det} and ``div-curl'' condition \eqref{div}, \eqref{curl}.
Furthermore, we assume that the initial data satisfies  the cone-condition, i.e.,
\begin{equation}\label{cone-condition}
supp \; \hat{u}_0 , supp \; \hat{E}_0 \subseteq \Omega_{\theta_0,\xi_0},
\end{equation}
for some vector $\xi_0$ and positive $\theta_0 \leq \frac{\pi}{2}$. Here $\hat{f}$ denotes the Fourier transform of function $f$.  If there exists a small number $c_0$ (only depending on $\mu$) and holds that,
\begin{equation}\label{small}
\theta_0 \big(  \|u_0\|_{H^2} + \|E_0\|_{H^2} \big) \leq c_0,
\end{equation}
then the system  admits a unique global smooth solution.
\end{theorem}

\subsection{Large Smooth Solution}

In this subsection, we show that the smooth solution constructed in Theorem 1.1 can be arbitrarily large.

If the Sobolev norm of initial data $\|u_0\|_{H^2}$ and $\|E_0\|_{H^2}$ is small enough, it immediately satisfies the assumption \eqref{small}.
However, the smallness of $\|u_0\|_{H^2}$ and $\|E_0\|_{H^2}$ is not necessary here. We then show that respectively.

{\bf {Construction of velocity $u_0$ with large Sobolev norm.}}
We should point out that this construction is already shown in the author's previous work \cite{zy}, which concerns the large solutions of 3D incompressible Navier-Stokes equations. For completeness, we shall give the details here.
First, let $f(x)$ be a smooth function in $\mathbb{R}^3$ with $\|f\|_{L^2} = M_0 < \infty$.
And assume that the support of $\hat f$ is restricted on a unit ball , i.e., $supp \; \hat f \subseteq B_1(0)$. Now we set,
\begin{equation}\nonumber
\hat v_{\lambda} (\xi)= \frac{i \xi \times \hat f(\xi - \lambda e_3)}{|\xi|}.
\end{equation}
Here $\lambda >1 $ is some positive parameter. Now, let's derive the Sobolev norm of $v_\lambda$. By direct calculation we have,
\begin{equation}\nonumber
\begin{split}
\|v_\lambda\|_{\dot H^{\frac{1}{2}}} = &\Big(\int_{\mathbb{R}^3} |\xi||\hat v_\lambda(\xi)|^2 \; d\xi\Big)^\frac{1}{2}
=  \Big(\int_{\mathbb{R}^3} \frac{|\xi \times \hat f(\xi - \lambda e_3)|^2}{|\xi|} \; d\xi\Big)^\frac{1}{2} \\
= & \Big(\int_{|\xi| \leq 1 } \frac{|(\xi + \lambda e_3) \times \hat f(\xi)|^2}{|\xi + \lambda e_3|} \; d\xi\Big)^\frac{1}{2}.
\end{split}
\end{equation}
It's obvious that $\|v_{\lambda}\|_{\dot H^\frac{1}{2}} \lesssim \lambda^\frac{1}{2}\|\hat f\|_{L^2}$. Also, we can achieve the lower bound control in the following,
\begin{equation}\label{low}
\begin{split}
\|v_\lambda\|_{\dot H^{\frac{1}{2}}} \gtrsim \;& \Big(\int_{|\xi| \leq 1 } \frac{\big | |\lambda e_3 \times \hat f(\xi)| - |\xi \times \hat f(\xi)| \big | ^2}{\lambda} \; d\xi\Big)^\frac{1}{2} \\
\gtrsim \;&  \Big(\int_{|\xi| \leq 1 } \Big |\lambda^{\frac{1}{2}} \big( |\hat f_1(\xi)|^2 +| \hat f_2(\xi)|^2 \big)^{\frac{1}{2}}- \lambda^{-\frac{1}{2}}|\xi \times \hat f(\xi)| \Big | ^2 \; d\xi\Big)^\frac{1}{2}.
\end{split}
\end{equation}
  It's easy to know that  $supp \; \hat v_{\lambda} \subseteq \Omega_{\theta_\lambda, e_3}$ with $\theta_{\lambda} = \arcsin{\frac{1}{\lambda}}$.
Therefore, if $\lambda$ is chosen large enough we shall get $\theta_\lambda \|v_{\lambda}\|_{\dot H^\frac{1}{2}} \ll 1$.
Moreover we can choose suitable $f$ ($f=(f_1,f_2,f_3)$ and $f_1\cdot f_2 \neq 0$ ), then the lower bound of $\|v_\lambda\|_{\dot H^{\frac{1}{2}}}$ \eqref{low} will be absolutely large (as $\lambda\rightarrow \infty$).

 We now construct the initial data for our system $u_0 = Re(v_{\lambda_0})$ (the real part of $v_{\lambda_0}$) and $\lambda_0$ is large enough, it then satisfies the assumption of Theorem \ref{t1.2}. Also  $\|u_0\|_{H^2}$  can be arbitrarily large since it's larger than $\|u_0\|_{\dot H^{\frac{1}{2}}}$.

{\bf {Construction of deformation tensor $E_0$ with large Sobolev norm}}. We shall choose  $E_0$ through the following system,
\begin{equation}\nonumber
\begin{cases}
U_t + u_0 \cdot \nabla U = \nabla u_0 U + \nabla u_0, \\
det(U + I) = 1, \\
U|_{t = 0} = 0,
\end{cases}
\end{equation}
here $u_0$ is some divergence free vector with cone-condition, i.e., $ supp \; \hat{u}_0 \subseteq \Omega_{\theta_0,\xi_0}$.
Notice the fact cone-condition will be preserved under multiplication operation, i.e., $supp \; \hat f, supp \; \hat g \subseteq \Omega_{\theta_0, \xi_0} \Longrightarrow supp  \widehat{f\cdot g} \subseteq \Omega_{\theta_0, \xi_0}$,  through Lemma \ref{lem1} we known that the solution $U(t, \cdot)$  satisfies the ``div-curl'' condition and cone-condition.
If we choose  $u_0$ with large Sobolev norm, the $\|U(T,\cdot)\|_{H^2}$ can also be large for some suitable $T$.
Then we can set $E_0(x) = U(T,x)$ which completes the construction.

\subsection{The Key of Proof}

In this subsection, we shall give a brief sketch to the key idea of proof.
As we have known, the lack of damping in the transport equation of $F$ causes the main  difficulty when deriving the global solutions especially for the large solutions.
Hence, there is few result concerning the global smooth solutions to system \eqref{e1.1} with large initial data.

In this paper, we use a brand new method to construct the large solutions satisfying ``cone-condition''.
And this cone-structure is well compatible with the nonlinear terms of system.
More precisely, the cone-structure and ``div-curl'' structure live in harmonious coexistence. They will be preserved together with the regularity of solution from the initial data.
The core of cone-structure is that under such structure one can move derivatives in nonlinear terms freely to an error term as needed.
And such an error term always contains one angle quantity related to the ``cone-condition''.
For more technical details, we refer to the next section.

\section{Energy Estimate}

Before deriving the energy estimate for Theorem \ref{t1.2}, we shall point out the assumptions for the initial data are natural and compatible with the system. We
give the correlative lemma in the following.

\begin{lemma}\label{lem1}
Assume that the assumptions for initial data \eqref{det}, \eqref{div}, \eqref{curl}, \eqref{cone-condition} are satisfied and $(u,F)$ is the solution of Cauchy problem \eqref{e1.1} with Hookean elasticity. Then the following equalities are always true,
\begin{equation}\nonumber
\begin{split}
\det(I+E)=&\;1, \\
\nabla \cdot E^T =&\;0,\\
\nabla_k E_{ij}-\nabla_j E_{ik}=&\;E_{lj} \nabla_l E_{ik}-E_{lk}\nabla_l E_{ij},\\
supp \; \hat{u} , supp \; \hat{E} \subseteq & \; \Omega_{\theta_0,\xi_0},
\end{split}
\end{equation}
for all time $t \geq 0 $.
\end{lemma}

\begin{proof}
The front three equalities can be understood as the consistency condition for changing of
variables between the Lagrangian and Eulerian coordinates. The proof
can be found in \cite{llzhou}. For the last property in this lemma we notice the fact: the cone-condition will be preserved  under multiplication operation.
\end{proof}

At the beginning of proof, we shall introduce the definition of energy.
Without loss of generality, we set $\mu = 1$ all through the paper. If $(u,E)$ is the solution to system \eqref{new} with the initial data assumptions in Theorem \ref{t1.2},
for any positive time $t > 0$ we set,

\begin{equation}\label{energy1}
\mathcal{E}_0(t) \doteq \sup_{0\leq \tau \leq t } \big( \theta_0^2 \| u\|_{H^2}^2 + \theta_0^2 \| E\|_{H^2}^2 \big) +\int_0^t \theta_0^2 \| \nabla u\|_{H^2}^2 d\tau ,
\end{equation}
and,
\begin{equation}\label{energy2}
\mathcal{E}_1(t)\doteq \int_0^t \theta_0^2 \| \nabla E\|_{H^1}^2 d\tau .
\end{equation}
In which $\theta_0$ is the coefficient given by the cone-condition of initial data \eqref{cone-condition}.
Moreover, we can define the whole energy as the sum of $\mathcal{E}_0(t)$ and $\mathcal{E}_1(t)$,

\begin{equation}\label{energy3}
\mathcal{E}_{total}(t) \doteq \mathcal{E}_0(t) + \mathcal{E}_1(t).
\end{equation}

In the next, we shall first give the \textit{a prior} estimate for $\mathcal{E}_0(t)$.
And for $\mathcal{E}_1(t)$, we can not obtain the \textit{a prior} estimate for it directly.
As a workaround, we turn to give the time integral estimate for $\| \nabla \cdot E \|_{H^1}$ instead of $\| \nabla  E \|_{H^1}$.
Somehow, it will imply the corresponding control for $\| \nabla  E \|_{H^1}$ in this problem.
For more details, we refer to Lemma \ref{lem2} in the next section.
~\\~

\textbf{Step 1: The estimate for $\mathcal{E}_0(t)$ }

~\\
As the first step, let's focus on the estimate for $\mathcal{E}_0(t)$.
The $L^2$ energy estimate for $\mathcal{E}_0(t)$ is trivial.
Taking $L^2$ inner product with $u$ for the first equation in \eqref{new}, taking $L^2$ inner product with $E$ for the second equation in \eqref{new}.
Summing them up and  adding the coefficient $\theta_0^2$ we shall obtain,
\begin{equation}\label{l2}
\frac{1}{2} \Big( \theta_0^2 \|u\|_{L^2}^2 + \theta_0^2 \|E\|_{L^2}^2 \Big) + \int_0^t \theta_0^2 \|\nabla u\|_{L^2}^2 d\tau \leq \frac{c_0^2}{2} .
\end{equation}

To control the highest order terms in $\mathcal{E}_0(t)$, we now apply $\nabla^2$ derivative to system \eqref{new}, take $L^2$ inner product with $\nabla^2 u$ and $\nabla^2 E$ respectively.
Summing them up and adding $\theta_0^2$ there is,

\begin{equation}\label{H2}
 \frac{1}{2} \frac{d}{dt} \Big( \theta_0^2 \|u\|_{\dot{H}^2}^2 + \theta_0^2 \|E\|_{\dot{H}^2}^2 \Big) +\theta_0^2 \|\nabla u\|_{\dot{H}^2}^2 = \sum_{i=1}^6 I_i ,
\end{equation}
where,
\begin{equation*}
\begin{split}
I_1 =&  \;-\theta_0^2 \big< \nabla^2 ( u\cdot\nabla u), \nabla^2 u \big>_{L^2},
\quad I_2 =\theta_0^2 \big< \nabla^2 \nabla\cdot( EE^T), \nabla^2 u \big>_{L^2}, \\
I_3 = & \;\theta_0^2 \big< \nabla^2 (\nabla\cdot E), \nabla^2 u \big>_{L^2},
\qquad \; I_4 =  -\theta_0^2 \big< \nabla^2 ( u\cdot\nabla E), \nabla^2 E \big>_{L^2}, \\
I_5 = & \; \theta_0^2 \big< \nabla^2 ( \nabla u E), \nabla^2 E \big>_{L^2},
\qquad I_6 =  \theta_0^2 \big< \nabla^2 \nabla u ,\nabla^2 E \big>_{L^2}.
\end{split}
\end{equation*}
Here we use the notation $\big<f_1\;,\;f_2\big>_{L^2}$ to represent the $L^2$ inner product of $f_1$ and $f_2$ on $\mathbb{R}^3$.
We then give the \textit{a prior} estimate for $I_1\sim I_6$ one by one.

Firstly, we would like to deal with the wildest and most difficult term $I_2$.
As mentioned before, the lack of  damping in the transport equation of $F$ causes the main  difficulty when deriving the global solutions especially for the large solutions.
Therefore, the nonlinear term containing $\nabla\cdot( EE^T)$ behaves wildly.
We can write this part into:

$$ [\nabla\cdot (EE^T)]_i= \nabla_j (EE^T)_{ij} = \nabla_j (E_{ik}E_{jk} )= \nabla_jE_{ik}E_{jk}+ E_{ik} \nabla_jE_{jk}. $$
Here and throughout the text, we omit the summation sign $\sum$ for repeated indicators.
Notice the condition \eqref{div} and Lemma \ref{lem1}, we have $\nabla_jE_{jk}=0$.
The second part on the right-hand side above is equal to zero then. For the first part, we can investigate its Fourier transform and write,

\begin{equation}\label{fourier}
\widehat{\nabla_jE_{ik}E_{jk}} = \vec{i} \int_{\mathbb{R}^3} \hat{E}_{jk}(\xi-\eta)\eta_j \hat{E}_{ik}(\eta) \; d\eta.
\end{equation}
Here we  split vector $\eta$ into two parts,
\begin{equation}\label{eta}
\begin{split}
\eta = \;&\frac{\eta \cdot (\xi-\eta)}{|\xi-\eta|} \frac{(\xi-\eta)}{|\xi-\eta|} + \Big( \eta-\frac{\eta \cdot (\xi-\eta)}{|\xi-\eta|} \frac{(\xi-\eta)}{|\xi-\eta|} \Big) \\
\doteq \;&\vec{a} +\vec{b}.
\end{split}
\end{equation}
Return back to \eqref{fourier}, it then yields:
\begin{equation}\label{E1}
 \widehat{\nabla_jE_{ik}E_{jk}} = \vec{i} \int_{\mathbb{R}^3} \vec{a_j}\hat{E}_{jk}(\xi-\eta) \hat{E}_{ik}(\eta) \; d\eta +\vec{i} \int_{\mathbb{R}^3} \vec{b_j}\hat{E}_{jk}(\xi-\eta) \hat{E}_{ik}(\eta) \; d\eta.
 \end{equation}
 Notice here $\partial_j E_{jk}=0$, it's obvious  that $\vec{a_j}\hat{E}_{jk}(\xi-\eta) =0$. Then the first part on the right-hand side of \eqref{E1} vanishes which implies,
 $$  \widehat{\nabla_jE_{ik}E_{jk}} = \vec{i} \int_{\mathbb{R}^3} \vec{b_j}\hat{E}_{jk}(\xi-\eta) \hat{E}_{ik}(\eta) \; d\eta . $$
A direct computations now yields,

\begin{equation}\label{bestimate}
\begin{split}
|b|= &\sqrt{|\eta|^2 + \frac{|\eta|^2|\xi-\eta|^2\cos^2 \theta |\xi-\eta|^2}{|\xi-\eta|^4}-2\frac{\eta\cdot(\xi-\eta)\eta\cdot(\xi-\eta)}{|\xi-\eta|^2}}\\
=& \sqrt{|\eta|^2 -|\eta|^2\cos^2\theta} \\
=& |\eta|\sin \theta,
\end{split}
\end{equation}
where $\theta$ is the included angle of vector $\eta$ and $\xi$. 
We can  now derive the bound control for \eqref{E1} in the following,
\begin{equation}
\begin{split}
  \big|\widehat{\nabla_jE_{ik}E_{jk}}\big| \leq  & \Big| \int_{\mathbb{R}^3} \vec{b_j}\hat{E}_{jk}(\xi-\eta) \hat{E}_{ik}(\eta) \; d\eta \Big| \\
  \leq & \int_{\mathbb{R}^3} |b|  \big|\hat{E}_{jk}(\xi-\eta)\big| \big| \hat{E}_{ik}(\eta) \big| \; d\eta .
\end{split}
\end{equation}
Notice the estimate for $|b|$, namely \eqref{bestimate}, it's natural to get
\begin{equation}\label{E2}
\begin{split}
\big|\widehat{\nabla_jE_{ik}E_{jk}}\big| \leq  \sum_j ( |\widehat{E_{jk}}| *   |\widehat{ \nabla E_{ik}}| )\sin \theta  \\
\leq \sum_j ( |\widehat{E_{jk}}|*|\widehat{\nabla E_{ik}}| ) \theta_0.
\end{split}
\end{equation}
Here we have used the cone-condition of initial data \eqref{cone-condition} and we use symbol $*$ to denote the convolution of functions. 

Now, let's give the estimate for $I_2$.
\begin{equation*}
\begin{split}
& \;\theta_0^2\big< \nabla^2 \nabla\cdot( EE^T), \nabla^2 u \big>_{L^2} \\
= & -\theta_0^2 \int_{\mathbb{R}^3} \nabla ( \nabla_j E_{ik}E_{jk}) \nabla^3 u_i \; dx \\
\leq & \; \theta_0^2 \| \nabla (\nabla_j E_{ik}E_{jk})\|_{L^2} \| \nabla^3 u \|_{L^2} \\
= & \;\theta_0^2 \big\| |\xi| \widehat{\nabla_j E_{ik}E_{jk}} \big\|_{L^2} \big\| \nabla^3 u \big\|_{L^2},
\end{split}
\end{equation*}
here $f^{\vee}$ denotes the inverse Fourier transform of function $f$.
Combining  with the bound control of $ \widehat{\nabla_jE_{ik}E_{jk}}$, namely \eqref{E1} and \eqref{E2}, it then implies

\begin{equation*}
\begin{split}
&  \;\theta_0^2 \big< \nabla^2 \nabla\cdot( EE^T) , \nabla^2 u \big>_{L^2} \\
\leq &\; \sum_j \theta_0^3 \| |\xi|( |\widehat{E_{jk}}|*|\widehat{\nabla E_{ik}}| ) \|_{L^2} \| \nabla^3 u_i \|_{L^2} \\
\leq &\; \sum_j \theta_0^3 \| \nabla (|\widehat{E_{jk}}|^{\vee} |\widehat{\nabla E_{ik}}|^{\vee}) \|_{L^2} \| \nabla^3 u_i \|_{L^2} \\
\leq & \; \sum_j \theta_0^3 \Big( \big\| \nabla |\widehat{\nabla E_{ik}}|^{\vee} \big\|_{L^2} \big\| |\widehat{E_{jk}}|^{\vee}\big\|_{L^{\infty}} + \big\|  |\widehat{\nabla E_{ik}}|^{\vee} \big\|_{L^4}\big\|  \nabla |\widehat{ E_{jk}}|^{\vee} \big\|_{L^4} \Big) \big\| \nabla^3 u_i\big\|_{L^2}\\
\lesssim & \;\theta_0^3 \Big( \|\nabla^2 E\|_{L^2} \|E\|_{H^2} + \|\nabla E\|_{H^1}^2 \Big) \big\|\nabla^3 u\big\|_{L^2}.
\end{split}
\end{equation*}
Naturally, we can get the time integral estimate for $I_2$. Recall the definition of energies \eqref{energy1} and \eqref{energy2} there is

\begin{equation}\label{I2}
\int_0^t |I_2| \; d \tau \lesssim  \mathcal{E}_0(t)\mathcal{E}_1^{\frac{1}{2}}(t).
\end{equation}
$$\;$$

The estimate for the left terms will follow the similar method above.
For $I_1$, we can write $[u\cdot\nabla u]_i = u_k \nabla_k u_i$ and obtain,
$$ \widehat{u_k \nabla_k u_i} = \vec{i}\int_{\mathbb{R}^3} \hat{u}_{k}(\xi-\eta)\eta_k \hat{u}_{i}(\eta) \; d\eta.$$
We still split $\eta$ into $ \vec{a} + \vec{b}$ where $\vec{a}$ and $\vec{b}$ are  defined  in \eqref{eta}. Notice the divergence free condition of velocity filed $\nabla \cdot u=0$  we  can write $ i\eta\cdot \hat{u}(\eta)=0$. Namely, $\vec{a_k}\hat{u}_{k}(\xi-\eta) =0$. The bound control for $\widehat{u_k\nabla_k u_i}$ is similar to \eqref{E2}, i.e.,
$$ \big|\widehat{u_k\nabla_k u_i} \big| \leq \sum_k \theta_0 ( |\widehat{u_k}| * |\widehat{\nabla u_i}|).$$
We then derive that,
\begin{equation*}
\begin{split}
& \; -\theta_0^2 \big< \nabla^2 (u\cdot\nabla u) , \nabla^2 u \big>_{L^2} \\
= & \; \theta_0^2 \int_{\mathbb{R}^3} \nabla ( u_k\nabla_k u_i) \nabla^3 u_i \; dx \\
\leq & \; \theta_0^2 \| \nabla (u_k \nabla_k u_i)\|_{L^2} \| \nabla^3 u_i \|_{L^2} \\
= & \;\theta_0^2 \big\| |\xi| \widehat{u_k \nabla_k u_i} \big\|_{L^2} \| \nabla^3 u_i\|_{L^2}.
\end{split}
\end{equation*}
Which implies,
\begin{equation*}
\begin{split}
& \; \theta_0^2 \big< \nabla^2 (u\cdot\nabla u) ,\nabla^2 u \big>_{L^2} \\
\leq & \; \sum_k \theta_0^3 \Big( \big\| \nabla |\widehat{\nabla_ku_i}|^{\vee} \big\|_{L^2} \big\| |\widehat{u_k}|^{\vee}\big\|_{L^{\infty}} + \big\|  \nabla |\widehat{u_k}|^{\vee} \big\|_{L^4}\big\|  |\widehat{\nabla_ku_i}|^{\vee} \big\|_{L^4} \Big) \big\| \nabla^3 u_i \big\|_{L^2}\\
 \lesssim & \;\theta_0^3 \Big( \|\nabla^2 u\|_{L^2} \|u\|_{H^2} + \|\nabla u\|_{H^1}^2 \Big) \big\|\nabla^3 u\big\|_{L^2}.
\end{split}
\end{equation*}
As the conclusion, we can get the time integral estimate for $I_1$ in the following,

\begin{equation}\label{I1}
\int_0^t |I_1| \; d \tau \lesssim  \mathcal{E}_0^{\frac{3}{2}}(t).
\end{equation}
$$\;$$

Next, we shall handle $I_3$ and $I_6$ at the same time. Using integration by parts, a direct computations yields the following cancellation.

\begin{equation}\label{I3I6}
\begin{split}
  I_3+I_6 = & \; \theta_0^2 \big< \nabla^2 \nabla\cdot E, \nabla^2 u \big>_{L^2} +\theta_0^2 \int_{\mathbb{R}^3} \nabla^2 \nabla_j u_i \nabla^2 E_{ij} \; dx \\
  = & \; \theta_0^2 \big< \nabla^2 \nabla\cdot E, \nabla^2 u \big>_{L^2}  - \theta_0^2\int_{\mathbb{R}^3} \nabla^2  u_i \nabla^2 \nabla_j E_{ij} \; dx \\
  =& \;\theta_0^2\int_{\mathbb{R}^3} \nabla^2 (\nabla\cdot E)_i \nabla^2 u_i \; dx -\theta_0^2 \int_{\mathbb{R}^3} \nabla^2 u_i \nabla^2 (\nabla\cdot E)_i \; dx \\
  =&\; 0.
\end{split}
\end{equation}

Now, let's focus on the term $I_4$. Following the process before, we can easily write $ [u\cdot \nabla E ]_{ij}= u_k\nabla_k E_{ij}$ and there is,

\begin{equation}\label{E3}
 \widehat{u_k\nabla_k E_{ij}} = \vec{i} \int_{\mathbb{R}^3} \hat{u}_k(\xi-\eta) \eta_k\hat{E}_{ij}(\eta) \; d\eta .
 \end{equation}
Like before, we can derive
$$\big|\widehat{u_k\nabla_k E_{ij}} \big| \leq \sum_k \theta_0 ( |\widehat{u_k}| * |\widehat{\nabla E_{ij}}|).$$
Before the detailed computation, we shall point out one fact:  the highest order term of $E$ vanishes here,
 $$ \theta_0^2 \big< u\cdot\nabla^3E, \nabla^2 E \big>_{L^2} = \theta_0^2\int_{\mathbb{R}^3} u\cdot\nabla \Big(\frac{|\nabla^2E|^2}{2} \Big)\; dx = -\theta_0^2\int_{\mathbb{R}^3} \nabla\cdot u \; \frac{|\nabla^2E|^2}{2}  \; dx =0 .$$
Therefore, we directly know here

\begin{equation*}
\begin{split}
& \; -\theta_0^2 \big< \nabla^2 (u\cdot\nabla E) , \nabla^2 E \big>_{L^2} \\
= & \; -\theta_0^2 \int_{\mathbb{R}^3} \nabla^2  ( u_k\nabla_k E_{ij}) \nabla^2 E_{ij} \; dx \\
\leq & \; \theta_0^2 \big\| \nabla^2 (u_k\nabla_k E_{ij})\big\|_{L^2} \big\| \nabla^2 E_{ij}\big\|_{L^2} \\
= & \;\theta_0^2 \big\| |\xi|^2 \widehat{u_k\nabla_k E_{ij}} \big\|_{L^2} \big\| \nabla^2 E_{ij}\big\|_{L^2}.
\end{split}
\end{equation*}
Combining the three equations above we then get,

\begin{equation*}
\begin{split}
& \;  \theta_0^2 \big< \nabla^2 (u\cdot\nabla E), \nabla^2 E \big>_{L^2} \\
\leq & \;\sum_k \theta_0^3 \Big( \big\| \nabla^2 |\widehat{u_k}|^{\vee} \big\|_{L^4} \big\| |\widehat{\nabla E_{ij}}|^{\vee}\big\|_{L^4} + \big\| \nabla |\widehat{u_k}|^{\vee} \big\|_{L^{\infty}} \big\| \nabla |\widehat{\nabla E_{ij}}|^{\vee}\big\|_{L^2} \Big) \big\| \nabla^2 E\big\|_{L^2}\\
 \lesssim & \;\theta_0^3 \Big( \|\nabla^2 u\|_{H^1} \| \nabla E\|_{H^1} + \|\nabla^2 u\|_{L^2}^{\frac{1}{2}} \|\nabla^3 u\|_{L^2}^{\frac{1}{2}}   \| \nabla^2E\|_{L^2}             \Big) \big\|\nabla^2 E\big\|_{L^2}.
\end{split}
\end{equation*}
It implies the following time integral estimate for $I_4$,
\begin{equation}\label{I4}
\int_0^t |I_4| \; d \tau \lesssim  \mathcal{E}_0(t)\mathcal{E}_1^{\frac{1}{2}}(t).
\end{equation}
$$\;$$

Finally, we turn to deal with the last term $I_5$. Like before, we write $ [\nabla u E]_{ik} = \nabla_j u_i E_{jk}$ and obtain that
$$\big|\widehat{\nabla_j u_i E_{jk}} \big| \leq \sum_j \theta_0 ( |\widehat{\nabla u_i }| *  |\widehat{E_{jk}}|),$$
by cone-condition. Using these known results, we directly derive that,

\begin{equation*}
\begin{split}
& \; \theta_0^2 \big<\nabla^2 (\nabla u E), \nabla^2 E \big>_{L^2} \\
= & \; \theta_0^2 \int_{\mathbb{R}^3} \nabla^2  ( \nabla_j u_i E_{jk}) \nabla^2 E_{ik} \; dx \\
\leq & \; \theta_0^2 \| \nabla^2 (\nabla_j u_i E_{jk})\|_{L^2} \| \nabla^2 E_{ik}\|_{L^2} \\
= & \;\theta_0^2 \big\| |\xi|^2 \widehat{\nabla_j u_i E_{jk}} \big\|_{L^2} \big\| \nabla^2 E_{ik}\big\|_{L^2}.
\end{split}
\end{equation*}
Indeed,
\begin{equation*}
\begin{split}
& \; \theta_0^2 \big<\nabla^2 (\nabla u E), \nabla^2 E \big>_{L^2} \\
\leq & \; \sum_j \theta_0^3 \Big( \big\| \nabla^2 |\widehat{\nabla  u_i}|^{\vee} \big\|_{L^2} \big\| |\widehat{E_{jk}}|^{\vee}\big\|_{L^{\infty}} + \big\| \nabla |\widehat{\nabla u_i}|^{\vee} \big\|_{L^4} \big\| \nabla |\widehat{E_{jk}}|^{\vee}\big\|_{L^4} \\
& + \big\| |\widehat{\nabla u_i}|^{\vee} \big\|_{L^{\infty}} \big\| \nabla^2|\widehat{E_{jk}}|^{\vee}\big\|_{L^2}\Big) \big\| \nabla^2 E_{ik}\big\|_{L^2}\\
 \lesssim & \;\theta_0^3 \Big( \|\nabla^3 u\|_{L^2} \|\nabla E\|_{L^2}^{\frac{1}{2}} \|\nabla^2 E\|_{L^2}^{\frac{1}{2}}+ \|\nabla^2 u\|_{H^1}\|\nabla E\|_{H^1}  \\
& + \|\nabla^2 u\|_{L^2}^{\frac{1}{2}} \|\nabla^3 u\|_{L^2}^{\frac{1}{2}}  \|\nabla^2 E\|_{L^2}  \Big) \big\|\nabla^2 E\big\|_{L^2}.
\end{split}
\end{equation*}
 It then implies the time integral control for $I_5$,

\begin{equation}\label{I5}
\int_0^t |I_5| \; d \tau \lesssim \mathcal{E}_0(t)\mathcal{E}_1^{\frac{1}{2}}(t).
\end{equation}

Now, we can give the estimate for $\mathcal{E}_0(t) $. Integral \eqref{H2} with time from $0\sim t$ and take all the estimates above into consideration. Namely, \eqref{I1}, \eqref{I2}, \eqref{I3I6}, \eqref{I4} and \eqref{I5}. We then obtain the following inequality,

\begin{equation}\nonumber
\begin{split}
& \frac{1}{2} \Big( \theta_0^2 \|u\|_{\dot{H}^2}^2 + \theta_0^2 \|E\|_{\dot{H}^2}^2 \Big) + \int_0^t \theta_0^2 \|\nabla u\|_{\dot{H}^2}^2 d\tau \\
 \lesssim & \; \mathcal{E}_0(0) +\mathcal{E}_0(t)\mathcal{E}_1^{\frac{1}{2}}(t).
\end{split}
\end{equation}
Combining this with \eqref{l2} and using Young's inequality, it then implies

\begin{equation}\label{l4}
\mathcal{E}_0(t) \lesssim \mathcal{E}_0(0)+\mathcal{E}_{total}^{\frac{3}{2}}(t)
\end{equation}
We recall here  $\mathcal{E}_{total}(t) \doteq \mathcal{E}_0(t)+\mathcal{E}_1(t)$ as defined in \eqref{energy3}.
 This completes the proof for $\mathcal{E}_0(t)$.
~\\~

As mentioned before, it's then natural to move on to the \textit{a prior} estimate for $ \| \nabla\cdot E \|_{H^1}$.
~\\~

\textbf{Step 2: The $H^1$ norm estimate for $\nabla \cdot E$ }
~\\

From the first equation of system \eqref{new} and the ``div'' structure assumption \eqref{div}, we have

\begin{equation*}
 \nabla\cdot E = u_t+ u\cdot\nabla u +\nabla P -\Delta u -\nabla\cdot (EE^T).
\end{equation*}
Taking $H^1$ inner product with $\nabla\cdot E$ on $\mathbb{R}^3$ and adding the coefficient $\theta_0^2$, it yields that

\begin{equation}\label{E}
\begin{split}
  & \;\theta_0^2 \| \nabla \cdot E\|_{H^1}^2 \\
  = & \;\theta_0^2 \big< \nabla \cdot E\; , \;u_t+ u\cdot\nabla u +\nabla P -\Delta u -\nabla\cdot (EE^T) \big>_{H^1} \\
   = &\; \theta_0^2 \frac{d}{dt} \big< u \;, \;\nabla\cdot E \big>_{H^1}  + \theta_0^2 \big<\nabla u\; ,  \;-u\cdot\nabla E + \nabla u E + \nabla u \big>_{H^1}  \\
  & + \theta_0^2 \big< \nabla \cdot E\; , \;u\cdot\nabla u +\nabla P -\Delta u -\nabla\cdot (EE^T) \big>_{H^1}\\
  =&\; \sum_{i=1}^8 J_i .
\end{split}
\end{equation}
Of course we should give the estimate for $J_1\sim J_8$ respectively. The time integral estimate for the first term $J_1$ is trivial.
\begin{equation}\label{J1}
\int_0^t |J_1| \; d\tau \leq \mathcal{E}_0(t).
\end{equation}
$$\;$$

For the next term $J_2$, we can follow the similar process as in $I_4$.
There is,

\begin{equation*}
\begin{split}
& \;\theta_0^2 \big< \nabla u \;,\; u\cdot\nabla E \big>_{H^1} \\
= & \;\theta_0^2 \int_{\mathbb{R}^3}  \nabla_j u_i(u_k\nabla_k E_{ij})           \; dx \\
\leq & \; \theta_0^2 \| \nabla_j u_i \|_{H^1} \| u_k\nabla_k E_{ij} \|_{H^1}  \\
 \lesssim  & \; \theta_0^3 \| \nabla u\|_{H^1}\|\nabla E\|_{H^1} \|\nabla u\|_{H^2}.
\end{split}
\end{equation*}
It implies,
\begin{equation}\label{J2}
\int_0^t |J_2| \; d\tau \lesssim \mathcal{E}_0^{\frac{3}{2}}(t).
\end{equation}
$$\;$$

The estimate for $J_3$ is similar to $I_5$ and there is,

\begin{equation*}
\begin{split}
\theta_0^2 \big< \nabla u \; , \; \nabla u E \big>_{H^1}
\leq & \; \theta_0^2 \| \nabla_j u_i \|_{H^1} \| \nabla_j u_i E_{jk} \|_{H^1}\\
\lesssim  &\; \theta_0^3 \| \nabla u\|_{H^1}\| E\|_{H^1} \|\nabla u\|_{H^2}.
\end{split}
\end{equation*}
It implies,
\begin{equation}\label{J3}
\int_0^t |J_3| \; d\tau \lesssim \mathcal{E}_0^{\frac{3}{2}}(t).
\end{equation}
$$\;$$

The estimate for linear term $J_4$ is directly shown by

\begin{equation}\label{J4}
\int_0^t |J_4| \; d\tau = \int_0^t \theta_0^2 \| \nabla u\|_{H^1}^2 \; d\tau\leq \mathcal{E}_0(t).
\end{equation}
$$\;$$

Next, we turn to deal with $J_5$. It can be handled like $I_1$. We derive,

\begin{equation*}
\begin{split}
\theta_0^2 \big< \nabla\cdot E \;,\; u\cdot\nabla u \big>_{H^1}
\leq & \; \theta_0^2 \| \nabla \cdot E_i\|_{H^1} \| u_k\nabla_k u_i  \|_{H^1} \\
\lesssim & \;  \theta_0^3 \|u\|_{H^2}\| \nabla u \|_{H^1} \|\nabla E\|_{H^1}.
\end{split}
\end{equation*}
And the time integral estimate is given by
\begin{equation}\label{J5}
\int_0^t |J_5| \; d\tau \lesssim \mathcal{E}_0(t)\mathcal{E}_1^{\frac{1}{2}}(t).
\end{equation}
$$\;$$

Using the ``div'' condition we shall derive $\nabla_i E_{ij}=0$. Therefore, it's obvious that
\begin{equation}\label{J6}
\begin{split}
  J_6 =& \;\theta_0^2 \big< \nabla\cdot E \;,\; \nabla P \big>_{H^1} = \theta_0^2 \sum_{k = 0}^1\int_{\mathbb{R}^3} \nabla^k \nabla_j E_{ij} \nabla^k \nabla_i P \; dx \\
  =& -\theta_0^2 \sum_{k = 0}^1 \int_{\mathbb{R}^3} \nabla^k \nabla_j \nabla_i E_{ij} \; \nabla^k P \; dx = 0.
\end{split}
\end{equation}
Also, the estimate for $J_7$ is trivial,

\begin{equation}\label{J7}
\int_0^t |J_7| \; d\tau \leq  \int_0^t \theta_0^2 \| \nabla E\|_{H^1} \| \nabla^2 u\|_{H^1} \; d\tau\leq \mathcal{E}_0^{\frac{1}{2}}(t)\mathcal{E}_1^{\frac{1}{2}}(t).
\end{equation}
$$\;$$

Finally we turn to handle the last term $J_8$. It can be dealt with the similar method like $I_2$.
We directly know that
\begin{equation*}
\begin{split}
\theta_0^2 \big< \nabla\cdot E  \;,\; \nabla\cdot (EE^T) \big>_{H^1}
\leq & \; \theta_0^2 \| \nabla \cdot E_{i,\cdot}\|_{H^1} \| \nabla_j E_{ik}E_{jk}  \|_{H^1}\\
\lesssim &\; \theta_0^3 \|E\|_{H^2}  \|\nabla E\|_{H^1}^2.
\end{split}
\end{equation*}
It implies,
\begin{equation}\label{J8}
\int_0^t |J_8| \; d\tau \lesssim \mathcal{E}_0^{\frac{1}{2}}(t)\mathcal{E}_1(t).
\end{equation}
$$\;$$

By this line, we have obtained the time integral estimate for $J_1\sim J_8$.
Namely, \eqref{J1}, \eqref{J2}, \eqref{J3}, \eqref{J4}, \eqref{J5}, \eqref{J6}, \eqref{J7} and \eqref{J8}.
Now, integral \eqref{E} with time from $0\sim t$ and take all the estimates above into consideration. Using Young's inequality we shall get

\begin{equation}\label{DIVE}
\int_0^t \theta_0^2 \|\nabla \cdot E \|_{H^1}^2 d\tau \lesssim \mathcal{E}_0(t) +\mathcal{E}_{total}^{\frac{3}{2}}(t).
\end{equation}
This completes the proof for the $H^1$ norm of $\nabla \cdot E$.

\section{Proof of Theorem \ref{t1.2}}

In Section 2, we have obtained the \textit{a priori} estimate for $\mathcal{E}_0(t)$ and $\nabla\cdot E$. However, what we need in $\mathcal{E}_1(t)$ is the time integral estimate for $\nabla E$.
At the beginning of this section, we shall announce that the estimate for $\nabla \cdot E$ implies the control for $\nabla E$.
After that, we give the proof for Theorem \ref{t1.2}.
Firstly, we introduce the following lemma.

\begin{lemma}\label{lem2}
Assume that the assumptions \eqref{div} and \eqref{curl} are satisfied,  $(u,E)$ is the solution of Cauchy problem for \eqref{new}. Then the following inequality is always true:
\begin{equation*}
  \| \nabla E \|_{H^1} \lesssim \| \nabla \cdot E\|_{H^1} + (\theta_0 \|E\|_{H^2}) \|\nabla E\|_{H^1},
\end{equation*}
for time $t \geq 0 $. Hence, for small quantity $\theta_0 \|E\|_{H^2} \ll 1$, it directly yields that
\begin{equation*}
\| \nabla E \|_{H^1} \lesssim \| \nabla \cdot E\|_{H^1}.
\end{equation*}
\end{lemma}

\begin{proof}

Notice the ``curl'' assumption \eqref{curl} in Theorem \ref{t1.2} and Lemma \ref{lem1}, we can derive that,
\begin{equation}\nonumber
\nabla_k E_{ij}-\nabla_j E_{ik}=E_{lj} \nabla_l E_{ik}-E_{lk}\nabla_l E_{ij}.
\end{equation}
Here we have used the definition for the curl of matrix $E$,
$$ [\nabla\times E]_{im}= [\nabla \times E_{i,\cdot}]_m = \delta_{mkj}\nabla_k E_{ij}.$$
It implies that

\begin{equation}\nonumber
\big|  \widehat{(\nabla\times E)_{im}} \big| \leq \sum_{j,k} \big(\big|  \widehat{E_{lj}\nabla_l E_{ik}} \big|+\big|  \widehat{E_{lk}\nabla_l E_{ij}} \big|\big).
\end{equation}
Following the process of \eqref{E2}, we have
\begin{equation}\nonumber
\begin{split}
\big|  \widehat{E_{lj}\nabla_l E_{ik}} \big| \lesssim& \theta_0 \sum_{l}|\hat{E}_{lj}| * |\widehat{\nabla E_{ik}}|, \\
\big|  \widehat{E_{lk}\nabla_l E_{ij}} \big| \lesssim& \theta_0 \sum_{l}|\hat{E}_{lk}| * |\widehat{\nabla E_{ij}}|,
\end{split}
\end{equation}
which implies,
$$ |\widehat{\nabla\times E}| \leq \theta_0 |\hat{E} | * |\widehat{\nabla E}|.$$
Here we use the symbol $\star$ to stand for convolution operation.
Therefore,
\begin{equation}\nonumber
\begin{split}
\| \nabla E\|_{H^1} \leq &\;\| \nabla \cdot E\|_{H^1}+\| \nabla \times E\|_{H^1} \\
\leq &\; \big\| \nabla \cdot E\big\|_{H^1} + \theta_0 \big\| |\hat{E}|^{\vee} |\widehat{\nabla E}|^{\vee} \big\|_{H^1} \\
\lesssim & \;\| \nabla \cdot E\|_{H^1} + (\theta_0 \| E\|_{H^2})\| \nabla E\|_{H^1}.
\end{split}
\end{equation}
It then completes the proof of this lemma.
\end{proof}

Now, let's give the proof for Theorem \ref{t1.2}.
Taking \eqref{DIVE} and Lemma \ref{lem2} into consideration, we then derive the \textit{a prior} estimate for $\mathcal{E}_1(t)$,

\begin{equation}\label{EEE}
\mathcal{E}_1(t)=  \int_0^t \theta_0^2 \|\nabla  E \|_{H^1}^2 d\tau \lesssim \mathcal{E}_0(t) +\mathcal{E}_{total}^{\frac{3}{2}}(t) + \mathcal{E}_0(t)\mathcal{E}_1.
\end{equation}

Multiplying \eqref{l4} and \eqref{EEE} by different suitable number, and summing them up, we can obtain the following inequality

\begin{equation}\label{total}
  \mathcal{E}_{total}(t) \leq C_1 \mathcal{E}_0(0) + C_1 \mathcal{E}_{total}^{\frac{3}{2}}(t) + C_1 \mathcal{E}_0(t)\mathcal{E}_1,
\end{equation}
here  $C_1$ is some positive constant.

According to the setting of initial data in Theorem \ref{t1.2}, there exists a positive number  $C_2$ such that $ \mathcal{E}_{total}(0)+C_1 \mathcal{E}_0(0) \leq C_2 c_0^2$. Due to the local existence theory which can be achieved through basic energy method, there exists a positive time $T$ such that

\begin{equation}\label{T}
  \mathcal{E}_{total}(t) \leq 2C_2 c_0^2, \qquad\qquad \forall t \in [0,T].
\end{equation}
Let $T^*$ be the largest possible time of $T$ for what \eqref{T} holds. Then we only need to show $T^* = \infty$. Notice the inequality \eqref{total}, we can use a standard continuation argument to show  the fact $T^* = \infty$ provided that $c_0$ is small enough. We omit some details here for convenience and  finish the proof of Theorem \ref{t1.2} then.

\section*{Acknowledgement}
 The  author is supported by Shanghai Sailing Program (18YF1405500), Fundamental Research Funds for the Central Universities (222201814026), China Postdoctoral Science Foundation (2018M630406, 2019T120308) and NSFC (11801175).

\end{document}